\documentclass[11pt,a4paper]{article}
 
\usepackage{fullpage}
\usepackage[english]{babel}
\usepackage[utf8]{inputenc}
\usepackage{amsmath,amsthm,amssymb,frcursive,mathrsfs,mathtools,stmaryrd}
\usepackage{graphicx,tikz,pgfplots,pgfplotstable}
\usepackage{color}
\usepackage{versions}

\pgfplotsset{every axis/.append style={
tick label style={font=\scriptsize}  
}}

\newtheorem{theorem}{Theorem}

\numberwithin{figure}{section}

\DeclareMathOperator{\Lip}{\Lip}

\DeclareMathOperator{\divg}{\bf div}

\newcommand{\cC}{{\cal C}}

\newcommand{\cE}{{\cal E}}

\newcommand{\cH}{{\cal H}}

\newcommand{\cS}{{\cal S}}

\newcommand{\C}{{\mathbb C}}

\newcommand{\N}{\mathbb{N}}
\newcommand{\R}{\mathbb{R}}

\newcommand{\td}{{\text d}}

\newcommand{\tin}{\text{ in }}
\newcommand{\ton}{\text{ on }}

\newcommand{\tapp}{\text{app}}

\newcommand{\bs}{\backslash}

\newcommand{\ph}{{\varphi}}
\newcommand{\g}[1]{\boldsymbol{#1}}
\newcommand{\norm}[2]{\left\|{#1}\right\|_{#2}}
\newcommand{\inner}[2]{\left<{#1}\right>_{#2}}

\makeindex
\title{Orthogonal modes of a fully anisotropic and heterogeneous elastic medium}
\author{Laurent Seppecher\thanks{Institut Camille Jordan, Ecole Centrale de Lyon \& UCBL, Lyon, F-69003, France}}

\begin{document}

\maketitle

\begin{abstract} The aim of this short note is to give a synthetic presentation of the mathematical elements that are used to solve the elastic wave system of equations in a bounded anisotropic elastic body, in a general framework. In particular, the proof of existence of a basis of orthogonal modes is given. We explain how these modes can by used to efficiently approach dynamic problems in time or harmonic regimes. 
\end{abstract}

The mathematical content of this note is an application of classical results from elliptic equation theory and from spectral theory of self-adjoint operators. Such elements can be found for instance in the well known books \cite{evans2022partial,brezis2011functional} or in French \cite{allaire2005analyse} for scalar elliptic equations. Even if the mathematics described in this note are classical, their presentation in a synthetic way, directly applicable to anisotropic and heterogeneous elastic bodies, may be useful for shortening the study of the dynamics of such media.

We consider a bounded elastic body $\Omega\in\R^3$ attached to fixed referential on a part of its boundary. We investigate the existence of an orthogonal basis of modes and recall how these modes can be used to efficiently approach the solutions of the static, harmonic and dynamic elastic problems.

\section{Notations}

The elastic body $\Omega$ is characterized by a fully anisotropic 4th-order elastic tensor $\g C_{ijk\ell}(x)$ and a density $\rho(x)$. Its boundary is denoted $\partial\Omega$. The space of the $3\times 3$ hermitian matrices is denoted $\cS_3(\C)$. The double dot product denotes the twice-contracted tensorial product. Depending on the situation we write 

\begin{equation}
(\g C:S)_{ij}=\sum_{k\ell}\g C_{ijk\ell}S_{k\ell},\qquad (S:\g C)_{k\ell}=\sum_{ij}S_{ij}\g C_{ijk\ell},\qquad S:T=\sum_{ij}S_{ij}	{T}_{ij}.
\end{equation}
for any $S,T\in S_3(\R)$. 

The space $H^1(\Omega)$ is the Sobolev space of real valued square integrable functions whose the gradient is also square integrable. The space $H^1(\Omega)^3$ is three dimensional vectorial valued version of $H^1(\Omega)$.

\section{General hypotheses}

The following hypotheses on the elastic medium are general and match both the classical framework of continuum mechanics and the mathematical theory of variational formulations that provides the well-posedness of the problem.  

We assume that the elastic body $\Omega$ is Lipschitz smooth (the boundary is locally a graph of a Lipschitz function). We also assume that this body is rigidly attached on a open part $\Gamma_{\text{Dir}}$ of its boundary $\partial \Omega$ to a fixed referential. We then make the three following assumptions on the elastic tensor $\g C$. 

It has bounded coefficients:
\begin{equation}
\g C\in L^\infty\left(\Omega,\R^{3^4}\right),
\end{equation}
it is symmetric on $\cS_3(\R)$:
\begin{equation}
\forall x\in \Omega,\quad\forall S,T\in \cS_3(\R),\quad S:\g C(x):T = T:\g C(x):S,
\end{equation}
and it is positive definite on $\cS_3(\R)$:
\begin{equation}\label{eq:coer}
\exists \alpha>0,\quad\forall x\in \Omega,\quad \forall S\in \cS_3(\R),\quad S:\g C(x):S \geq \alpha \norm{S}{2}^2.
\end{equation}
Remark that such tensor is characterized by $21$ independent coefficients in dimension $3$. We also assume that the density is bounded by above and by below by a positive value. 
\begin{equation}
\rho\in L^\infty(\Omega)\qquad\text{and}\qquad\exists \beta>0,\quad \forall x\in\Omega,\quad \rho(x)\geq \beta. 
\end{equation}

\section{The elasto-static problem}

The body is fixed on  $\Gamma_{\text{Dir}}$ and some boundary stress $\g g$ is applied on $\Gamma_{\text{Neu}}:=\partial \Omega\bs {\Gamma_{\text{Dir}}}$. A internal stress $\g f$ is also applied inside $\Omega$. Hence the displacement $\g u$ satisfies the following system of equations/boundary conditions:
\begin{equation}\label{eq:static}
\left\{\begin{aligned}
-\divg(\g C:\cE(\g u)) &= \g f\quad\tin\Omega\\
\g u &=\g 0\quad\ton\Gamma_\text{Dir}\\
(\g C:\cE(\g u))\cdot \g n &= \g g\quad\ton\Gamma_\text{Neu}.
\end{aligned}\right.
\end{equation}

\begin{theorem} Define the space $\cH:=\{\g u\in H^1(\Omega)^3\ |\ \g u=\g 0 \ton \Gamma_{\text{Dir}}\}$. If $\g f\in L^2(\Omega)^3$ and $\g g\in L^2(\Gamma_{\text{Neu}})^3$, then the problem \eqref{eq:static} admits a unique solution $\g u$ in $\cH$. Moreover, this solution satisfies the energy estimate
\begin{equation}\nonumber
\norm{\cE(\g u)}{L^2(\Omega)}\leq c_1\norm{\g f}{L^2(\Omega)} + c_2\norm{\g g}{L^2(\Gamma_{\text{Neu}})}
\end{equation}
where $c_1$ and $c_2$ are constants depending only on $\Omega,\Gamma_{\text{Neu}}$ and $\alpha$ from \eqref{eq:coer}. In the case $\g g=\g 0$, this problem defines a bounded linear operator $A:\g f\in L^2(\Omega)\mapsto \g u\in \cH$. 
\end{theorem}

\begin{proof} The space $\cH$ is endowed with the inner product 
\begin{equation}\nonumber
\inner{\g u,\g v}{\cH}:=\int_\Omega\cE(\g u):\cE({\g v})
\end{equation}
and the associated norm $\norm{\g u}{\cH}:=\inner{\g u,\g u}{\cH}^{1/2}=\norm{\cE(\g u)}{L^2(\Omega)}$. This is indeed a norm (equivalent to the $H^1$-norm) thanks to Poincaré (constant $c_{\text{P}}$) and Korn (constant $c_{\text{K}}$) inequalities:
\begin{equation}\nonumber
\forall \g u\in\cH,\qquad\norm{\g u}{H^1(\Omega)}\leq c_{\text{P}}\norm{\nabla\g u}{L^2(\Omega)}\leq c_{\text{P}}c_{\text{K}} \norm{\cE(\g u)}{L^2(\Omega)}= c_{\text{P}}c_{\text{K}} \norm{\g u}{\cH}. 
\end{equation}
The variational formulation of $\eqref{eq:static}$ reads
\begin{equation}\label{eq:fv}
\int_\Omega(\g C:\cE(\g u)):\cE({\g v})=\int_\Omega\g f\cdot{\g v} + \int_{\Gamma_\text{Neu}}\g g\cdot {\g v},\qquad \forall \g v\in \cH.
\end{equation}
The left-hand side bilinear form $a(\g u,\g v)$ is continuous in $\cH$:
\begin{equation}\nonumber
|a(\g u,\g v)|\leq \norm{\g C}{L^\infty(\Omega)}\norm{\g u}{\cH}\norm{\g v}{\cH}.
\end{equation}
The right-hand side linear form $L\g v$ is continuous in $\cH$:
\begin{equation}\nonumber\begin{aligned}
|L\g v| &\leq \norm{\g f}{L^2(\Omega)}\norm{\g v}{L^2(\Omega)} + \norm{\g g}{L^2(\Gamma_\text{Neu})}\norm{\g v}{L^2(\Gamma_\text{Neu})}\\
&\leq  \norm{\g f}{L^2(\Omega)}\norm{\g v}{H^1(\Omega)} + c_T\norm{\g g}{L^2(\Gamma_\text{Neu})}\norm{\g v}{H^1(\Omega)}\\
&\leq  c_{\text{P}}c_{\text{K}}\left(\norm{\g f}{L^2(\Omega)} + c_T\norm{\g g}{L^2(\Gamma_\text{Neu})}\right)\norm{\g v}{\cH}.
\end{aligned}\end{equation}
Where $c_\text{T}$ is the norm of the trace operator $T:\cH\to L^2(\Gamma_\text{Neu})$. Moreover $a$ is coercive in $\cH$:
\begin{equation}\nonumber
|a(\g u,\g u)| =  \int_\Omega \cE(\g u):\g C:\cE({\g u})\geq \alpha \int_\Omega |\cE(\g u)|^2\geq \alpha \norm{\g u}{\cH}^2. 
\end{equation}
Thanks to Lax-Milgram theorem, the variational formulation \eqref{eq:fv} admits a unique solution $\g u\in\cH$. Moreover, it satisfies 

 \begin{equation}\nonumber
 \begin{aligned}
  \alpha\norm{\g u}{\cH}^2 &\leq a(\g u,\g u)\leq |L\g u|\leq c_{\text{P}}c_{\text{K}}\left(\norm{\g f}{L^2(\Omega)} + c_T\norm{\g g}{L^2(\Gamma_\text{Neu})}\right)\norm{\g u}{\cH}.\\
    \norm{\g u}{\cH} &\leq \frac{c_{\text{P}}c_{\text{K}}}\alpha\left(\norm{\g f}{L^2(\Omega)} + c_T\norm{\g g}{L^2(\Gamma_\text{Neu})}\right).
 \end{aligned}
 \end{equation}
\end{proof}

\section{The elasto-dynamic problem}

 In the same setting, we look for elasto-dynamic solutions $\g U(x,t)$ that satisfies the elastic wave system of equations/boundary conditions 
\begin{equation}\label{eq:dynamic}
\left\{\begin{aligned}
\rho\partial_{tt}\g U-\divg(\g C:\cE(\g U)) &= \g F\quad\tin\Omega\times \R\\
\g U &=\g 0\quad\ton\Gamma_\text{Dir}\times \R\\
(\g C:\cE(\g U))\cdot \g n &= \g G\quad\ton\Gamma_\text{Neu}\times \R.
\end{aligned}\right.
\end{equation}
where the sources $\g F(x,t)$ and $\g G(x,t)$ can varry in space and time. We may look for time solutions under the frequency decomposition (Fourier) form 
\begin{equation}\nonumber
\g U(x,t)=\int_\R \g u(x,\omega)e^{i\omega t}\td \omega.
\end{equation}
We then write the sources as
 \begin{equation}\nonumber
\g F(x,t)=\int_\R \g f(x,\omega)e^{i\omega t}\td \omega\qquad\text{and}\qquad\g G(x,t)=\int_\R \g g(x,\omega)e^{i\omega t}\td \omega.
\end{equation}
Then for each fixed frequency $\omega\in\R$, $\g u(.,\omega)$ satisfies the harmonic wave system of equations
\begin{equation}\label{eq:harm}
\left\{\begin{aligned}
\divg(\g C:\cE(\g u))+\rho\, \omega^2\g u&= -\g f\quad\tin\Omega\\
\g u &=\g 0\quad\ton\Gamma_\text{Dir}\\
(\g C:\cE(\g u))\cdot \g n &= \g g\quad\ton\Gamma_\text{Neu}.
\end{aligned}\right.
\end{equation}

The specific case $\g f=\g 0$ and $\g g=\g 0$ is called the homogenous harmonic problem and, as we will see, it is an eigenvalue problem. The following theorem states the existence of an orthogonal modal basis in a weighed space depending on the density.

\begin{theorem} There exists a unique non decreasing sequence $(\lambda_k)_{k\in\N}$ satisfying $\lambda_0>0$ and $\lim_{n\to +\infty}\lambda_k=+\infty$ and there exists a sequence of eigenmodes $(\g u_k)_{k\in\N}$ in $\cH$ such that 
\begin{equation}\nonumber
\forall k\in\N,\qquad\left\{\begin{aligned}
\divg(\g C:\cE(\g u_k)) + \rho \lambda_k\g u_k &= \g 0\quad\tin\Omega\\
\g u_k &=\g 0\quad\ton\Gamma\\
(\g C:\cE(\g u_k))\cdot \g n &= \g 0\quad\ton\partial\Omega \bs \Gamma.
\end{aligned}\right.
\end{equation}
The sequence $(\g u_k)_{k\in\N}$ is complete in $L^2(\Omega)^3$ and orthonormal in the sense
\begin{equation}\nonumber
\int_\Omega\rho\, \g u_k\cdot {\g u}_\ell = \delta_{k\ell},\qquad\forall k,\ell\in\N. 
\end{equation}
Moreover, the eigenmodes $\g u_k$ are in $\cC^\infty(\Omega)^3$ and even in  $\cC^\infty(\overline{\Omega})^3$ if $\Omega$ is smooth. 
\end{theorem}

\begin{proof} We first define a new Hilbert space 
\begin{equation}\nonumber
L^2_\rho(\Omega)^3:=\left\{\g u:\Omega\to\R^3\ |\ \int_\Omega\rho\, |\g u|^2<+\infty\right\}
\end{equation}
endowed with the inner product
\begin{equation}\nonumber
\inner{\g u,\g v}{\rho}:=\int_\Omega\rho\, \g u\cdot{\g v}.
\end{equation}
This Hilbert space is isomorphic to the classic Hilbert space $L^2(\Omega)^3$ as $0<\beta\leq \rho\leq \norm{\rho}{L^\infty(\Omega)}<+\infty$. We define now the linear operator $A_\rho$ by 
\begin{equation}
A_\rho\g f:=A(\rho\, \g f)
\end{equation}
where $A:L^2(\Omega)\to \cH$ is the bounded linear operator given by Theorem 1. This new operator satisfies the following properties: 

1. $A_\rho:L^2_\rho(\Omega)^3\to L^2_\rho(\Omega)^3$ is bounded. Indeed,
\begin{equation}\nonumber
\begin{aligned}
\norm{A_\rho(\g f)}{L^2_\rho(\Omega)} &=\norm{\sqrt{\rho}A(\rho\, \g f)}{L^2(\Omega)}\leq \norm{\rho}{L^\infty(\Omega)}^\frac12\norm{A(\rho\, \g f)}{L^2(\Omega)}\leq c_1\norm{\rho}{L^\infty(\Omega)}^\frac12\norm{\rho\, \g f}{L^2(\Omega)}\\
& \leq c_1\norm{\rho}{L^\infty(\Omega)}\norm{\sqrt{\rho}\, \g f}{L^2(\Omega)} \leq c_1\norm{\rho}{L^\infty(\Omega)}\norm{\g f}{L^2_\rho(\Omega)}. 
\end{aligned}
\end{equation}

2. The operator $A_\rho$ is self-adjoint in $L^2_\rho(\Omega)^3$. Indeed, consider $\g f\in L^2_\rho(\Omega)^3$ and call $\g u:=A_\rho\g f$. It satisfies
\begin{equation}\nonumber
a(\g u,\g v) = \int_\Omega\rho\, \g f\cdot{\g v},\qquad\forall \g v\in\cH.
\end{equation}
Then for any $\g \ph\in L^2_\rho(\Omega)^3$,
\begin{equation}\nonumber
\begin{aligned}
\inner{A_\rho\g f,\g \ph}\rho &= \inner{\g u,\g \ph}\rho = \int_\Omega\rho\, \g u\cdot {\g \ph} = {a(A_\rho\g \ph,\g u)}  = a(\g u,A_\rho\g \ph) = \int_\Omega\rho\, \g f\cdot {A_\rho\g \ph}=\inner{\g f,A_\rho\g \ph}\rho.
\end{aligned}
\end{equation}

3. The operator $A_\rho:L^2_\rho(\Omega)^3\to L^2_\rho(\Omega)^3$ is compact. Indeed, from the Rellish theorem, the canonical injection $I:H^1(\Omega)\hookrightarrow L^2(\Omega)$ is compact. By equivalence of topologies $I:\cH\hookrightarrow L^2_\rho(\Omega)^3$ is also compact. Hence the operator $A_\rho:L^2_\rho(\Omega)^3\to L^2_\rho(\Omega)^3$ is compact as the composition of $A_\rho:L^2_\rho(\Omega)^3\to \cH$ continuous and $I:\cH\hookrightarrow L^2_\rho(\Omega)^3$ compact. 

From the spectral theory of self-adjoint compact operators (see \cite{brezis2011functional} chapter $6$), $A_\rho$ admits an orthogonal spectral decomposition: there exists a unique positive decreasing sequence of eigenvalues $(\mu_n)_{n\in\N}$ of $\R^+$ such that $\mu_n\to 0$ when $n\to +\infty$, and there exists an orthogonal Hilbert basis of eigenvectors $(\g f_n)_{n\in\N}$ of $L^2_\rho(\Omega)^3$ such that 
\begin{equation}\nonumber
A_\rho\g f_n=\mu_n\g f_n,\qquad \forall n\in\N.
\end{equation}
Call now $\lambda_n:=1/\mu_n$ and $\g u_n:=A_\rho\g f_n/\norm{A_\rho\g f_n}{L^2_\rho(\Omega)}\in\cH$. For all $n\in\N$, $\g u_n$ satisfies 
\begin{equation}\nonumber
-\divg(\g C:\cE(\g u_n)) = \rho\frac{\g f_n}{\norm{A_\rho\g f_n}{L^2_\rho(\Omega)}}=\rho\lambda_n\frac{A_\rho\g f_n}{\norm{A_\rho\g f_n}{L^2_\rho(\Omega)}}=\rho\lambda_n\g u_n.
\end{equation}
Moreover, $\norm{\g u_n}{L^2_\rho(\Omega)} = 1$ and for any $n\neq m$, $\int_\Omega\rho\, \g u_n\cdot {\g u_m}=0$ because
\begin{equation}\nonumber
\begin{aligned}
\int_\Omega\rho\, A_\rho\g f_n\cdot {A_\rho\g f_m} = \mu_n\mu_m \int_\Omega\rho\,  \g f_n\cdot {\g f_m} = \inner{\g f_n\cdot \g f_m}\rho = 0.
\end{aligned}
\end{equation}
The smoothness of the eigenmodes $\g u_n$ is showed be recurrence and with the use of the Sobolev injection theorem. 
\end{proof}

\section{Approximation of the solution using the modes}

If the modes are computed in advance, they can be used to efficiently approach any solution of the elastic system by computing only few first coefficients of the modal decomposition of the solution. 

\subsection{Elasto-static solution}
Let $\g u\in \cH$ be a solution of the elastostatic  system with sources $\g f$ and $\g g$. We can decompose $\g u$ on the eigenmodes:
\begin{equation}
\g u=\sum_{n\in\N}\alpha_n\g u_n
\end{equation}
where $(\alpha_n)\in \ell^2(\N)$. Hence 
\begin{equation}\nonumber
\begin{aligned}
\sum_{n\in\N}\alpha_n\int_\Omega(\g C:\cE(\g u_n)):\cE(\g v) &=\int_\Omega\g f\cdot\g v + \int_{\Gamma_\text{Neu}}\g g\cdot \g v,\qquad \forall \g v\in \cH\\
\sum_{n\in\N}\alpha_n\lambda_n\int_\Omega \rho\, \g u_n\cdot \g v &=\int_\Omega\g f\cdot\g v + \int_{\Gamma_\text{Neu}}\g g\cdot \g v,\qquad \forall \g v\in \cH.
\end{aligned}
\end{equation}
Choosing $\g v={\g u_n}$ we get from the orthogonallity of the modes in $L^2_\rho(\Omega)^3$,
\begin{equation}\nonumber
\alpha_n\lambda_n=\int_\Omega\g f\cdot{\g u_n} + \int_{\Gamma_\text{Neu}}\g g\cdot {\g u_n},\qquad \forall n\in\N.
\end{equation}
If we call $f_n:=\int_\Omega\g f\cdot{\g u_n}$ and $g_n:=\int_{\Gamma_\text{Neu}}\g g\cdot {\g u_n}$, we get the formula 
\begin{equation}
\alpha_n=\frac{f_n+g_n}{\lambda_n}. 
\end{equation}

In a practical point of view, if one has knowledge of the modes $(\lambda_n,\g u_n)$, given $\g f$ and $\g g$, we fix a maximum mode number $N\in\N$ and compute the inner products $f_0,\dots,f_N$, $g_0,\dots,g_N$ and the corresponding coefficients $\alpha_0,\dots,\alpha_N$. The approximated solution is computed as the truncation of the series
\begin{equation}\nonumber
\g u^\tapp(x):=\sum_{n=0}^N\alpha_n\g u_n(x). 
\end{equation}
An $L^2$-error bound can be computed. It depends both on the smoothness of $\g f$ and $\g g$ and the increase speed of the eigenvalues sequence.

\subsection{Harmonic solutions}

Let $\g u(.,\omega)\in \cH$ be a solution of the of the harmonic system with sources $\g f$ and $\g g$. We can decompose $\g u$ in space using the eigenmodes:
\begin{equation}\label{eq:modharm}
\g u(x,\omega)=\sum_{n\in\N}\alpha_n(\omega)\g u_n(x).
\end{equation}
We then write 
\begin{equation}\nonumber
\begin{aligned}
\sum_{n\in\N}\alpha_n(\omega)\int_\Omega(\g C:\cE(\g u_n)):\cE(\g v)-\omega^2\alpha_n(\omega)\int_\Omega \rho\, \g u_n\cdot \g v &=\int_\Omega\g f(.,\omega)\cdot\g v + \int_{\Gamma_\text{Neu}}\g g(.,\omega)\cdot \g v,\qquad \forall \g v\in \cH\\
\sum_{n\in\N}\alpha_n(\omega)(\lambda_n-\omega^2)\int_\Omega \rho\, \g u_n\cdot \g v &=\int_\Omega\g f(.,\omega)\cdot\g v + \int_{\Gamma_\text{Neu}}\g g(.,\omega)\cdot \g v,\qquad \forall \g v\in \cH.
\end{aligned}
\end{equation}
Choosing $\g v={\g u_n}$ we get from the orthogonallity of the modes:
\begin{equation}\nonumber
\alpha_n(\omega)(\lambda_n-\omega^2)=\int_\Omega\g f(.,\omega)\cdot{\g u_n} + \int_{\Gamma_\text{Neu}}\g g(.,\omega)\cdot {\g u_n},\qquad \forall n\in\N.
\end{equation}
If we call $f_n(\omega):=\int_\Omega\g f(.,\omega)\cdot{\g u_n}$ and $g_n(\omega):=\int_{\Gamma_\text{Neu}}\g g(.,\omega)\cdot {\g u_n}$, we get the formula 
\begin{equation}\label{eq:alphan2}
\alpha_n(\omega)=\frac{f_n(\omega)+g_n(\omega)}{\lambda_n-\omega^2}. 
\end{equation}
We clearly see the problem when $\omega^2$ is equal to one eigenvalue $\lambda_n$. If $\omega^2\notin\{\lambda_n\ |\ n\in\N\}$ it can be shown that the problem \eqref{eq:harm} admits a unique solution in $\cH$ and this solution is given by formulas \eqref{eq:modharm} and \eqref{eq:alphan2}.

\subsection{ Dynamic solutions}
Let $\g U(.,t)\in \cH$ be a solution of the of the time wave system of equations with sources $\g f$ and $\g g$. This solution can be approached by Fourier recomposition mode-by-mode through the formula

\begin{equation}
\g U(x,t)=\sum_{n\in\N}\int_\R \alpha_n(\omega)e^{i\omega t}\td \omega\, \g u_n(x),\qquad \forall x\in\Omega,\ t\in\R,
\end{equation}
where $\alpha_n(\omega)$ is given by \eqref{eq:alphan2}.

\bibliography{biblio.bib}

\end{document}